\documentclass[11pt]{article}
\usepackage{amsmath,amsfonts,amssymb,amsthm,epsfig,epstopdf,titling,url,array, pdflscape, fancyhdr, lastpage, tikz, hyperref}
\usepackage{blindtext}
\usepackage{dsfont}
\usepackage{caption}
\usepackage{subcaption}

\title{Skew Axial Algebras of Monster Type II}
\author{Michael Turner\textsuperscript{}\thanks{\textsuperscript{}
School of Mathematics, University of Birmingham, Edgbaston, Birmingham, B15 2TT, UK, Email:
mxt187@student.bham.ac.uk
}}
\date{\today}

\newtheoremstyle{sltheorem}
{}                
{}                
{\slshape}        
{}                
{\bfseries}       
{.}               
{ }               
{}                

\theoremstyle{sltheorem}
\newtheorem{thm}{Theorem}[section]
\newtheorem{lem}[thm]{Lemma}
\newtheorem{prop}[thm]{Proposition}

\newtheorem*{thm*}{Theorem}
\newtheorem*{cor*}{Corollary}

\theoremstyle{definition}
\newtheorem{defn}[thm]{Definition}

\newtheorem{exmp}[thm]{Example}

\newtheorem{no}[thm]{Notation}

\theoremstyle{remark}
\newtheorem*{rem}{Remark}
\newtheorem*{note}{Note}

\newcommand{\id}{\mathds{1}}
\newcommand{\F}{\mathbb{F}}
\newcommand{\mcal}[1]{\mathcal{#1}}

\newcommand{\Mod}[1]{\ (\mathrm{mod}\ #1)}

\newcommand{\GenG}[1]{\langle #1\rangle}
\newcommand{\GenA}[1]{\langle\!\langle #1\rangle\!\rangle}
\newcommand{\fustar}[1]{(\mcal{#1},\star)}

\newcommand{\Z}{\mathbb{Z}}
\newcommand{\N}{\mathbb{N}}

\newcommand{\spec}[1]{\text{Spec}(#1)}

\newcommand{\al}{\alpha}
\newcommand{\bt}{\beta}

\newcommand{\lm}{\lambda}

\newcommand{\ep}{\epsilon}

\newcommand{\tu}[1]{\tau_{#1}}
\newcommand{\gm}{\gamma}

\newcommand{\jor}[1]{\mcal{J}(#1)}
\newcommand{\mon}[1]{\mcal{M}(#1)}
\newcommand{\Spec}[1]{\text{Spec}(#1)}

\newcommand{\hw}{\mcal{H}}
\newcommand{\hwc}{\hat{\mcal{H}}}
\newcommand{\TB}{2\mathrm{B}}
\newcommand{\TC}{3\mathrm{C}}
\newcommand{\TA}{3\mathrm{A}}
\newcommand{\FA}{4\mathrm{A}}
\newcommand{\FB}{4\mathrm{B}}
\newcommand{\FJ}{4\mathrm{J}}
\newcommand{\FY}{4\mathrm{Y}}
\newcommand{\CA}{5\mathrm{A}}
\newcommand{\SA}{6\mathrm{A}}
\newcommand{\SJ}{6\mathrm{J}}
\newcommand{\SY}{6\mathrm{Y}}
\newcommand{\IY}[1]{\mathrm{IY}_{#1}}

\begin{document}

\maketitle
\begin{abstract}
In our first paper, we looked at $2$-generated primitive axial algebras of Monster type with skew axet $X'(1+2)$. We continue our work by focusing on larger skew axets and classifying all such algebras with skew axets. This brings us one step closer to a complete classification of all $2$-generated primitive axial algebras of Monster type. 
\end{abstract}
\section{Introduction}

Axial algebras are a class of non-associative commutative algebras which are strongly related to groups; in particular, the sporadic groups and the $3$-transposition groups. Hall, Rehren and Shpectorov defined axial algebras in \cite{hall2015universal} and \cite{hall2015primitive} as they drew inspiration from Majorana algebras in \cite{ivanov2009monster} along with previous work around the Griess algebra. The $2$-generated Majorana algebras were classified in \cite{ivanov2010majorana}, using results of Norton and Sakuma. In \cite{rehren2015axial}, Rehren investigated axial algebras of Monster type $(\al,\bt)$ and found generalisation of the Norton-Sakuma algebras. Since then, other algebras have been found in \cite{joshi2020axial}, \cite{franchi2022infinite}, and \cite{yabe2023classification}  to name a few. 

In \cite{mcinroy2023forbidden}, M\textsuperscript{c}Inroy and Shpectorov generalised the term of shape used by Ivanov for Majorana algebras to identify each $2$-generated axial subalgebra with axes in the axet. If one has an axial algebra generated by $n>2$ axes, this is a key tool to understand the structure of the algebra. In their work, they noticed for $2$-generated $C_2$-axets, only two options occur; they are either regular or skew. There are many examples of axial algebras with regular axets however at the time of their first preprint, there were no know examples of a skew axial algebra.

In \cite{turner2023skew}, we produced examples of skew axial algebras of Monster type, namely $\TC(\al,1-\al)$, $Q_2(\frac{1}{3},\frac{2}{3})$, and $Q_2(\frac{1}{3})^\times \oplus \GenG{\id}$, moreover proved that these are the only algebras with axet $X'(1+2)$. Our work focused on $k=1$ however for $k\in \N$ and $k\geq 2$, it was still an open problem wherever any primitive $2$-generated $\mon{\al,\bt}$-axial algebras with axet $X'(k+2k)$ existed. In this paper, we will be concerned with those algebras and complete the  classification of primitive $2$-generated axial algebras of Monster type with skew axet. This is one step further to classifying all $2$-generated primitive axial algebras of Monster type.

We proceed as follows: Section 2 will be reminding the reader of definitions of axial algebras, axets and the construction of a $2$-generated axial algebra. This section will be kept to the bare minimum and we recommend the reader to look at the previous paper for a more detailed introduction. It is important to note that we change our approach to the initial paper. In \cite{turner2023skew}, we found an upper bound of the dimension and we used GAP to produce relations between the indeterminants before looking at the structure of the possible algebras. However in this paper, there are no such calculations as we focus on subalgebras and the structure from the beginning. We will split the task into two separate cases. When $k$ is odd, we will use the idea mentioned in our first paper. Since $X'(k+2k)$ has a subaxet of $X'(1+2)$, we can apply results from the smaller case. Moreover, there is a symmetric subalgebra in the possible algebra. This leads to investigating a few possible algebras before getting the following result.
\begin{thm}
Let $\F$ have characteristic not equal to $2$ with $k\geq 3$ and odd. There are no primitive $2$-generated $\mon{\al,\bt}$-axial algebras over $\F$ with axet $X'(k+2k)$.
\end{thm}

When $k$ is even, we will proceed with new ideas.  First, we will show that $X'(k+2k)$ has a subaxet $X'(2^n+2^{n+1})$ for some $n \in \N$ and so restricts our problem. Investigating possible algebras with axet $X'(2^n+2^{n+1})$, we will show that no algebra exists before extending it back to $k$. This produces the following theorem.
\begin{thm}
Let $\F$ have characteristic not equal to $2$ with $k\geq 2$ and even. There are no primitive $2$-generated $\mon{\al,\bt}$-axial algebra over $\F$ with axet $X'(k+2k)$.
\end{thm}
We therefore have the complete classification of the primitive $2$-generated skew axial algebras of Monster type. 
\begin{thm}
Let $\F$ have characteristic not equal to $2$. Suppose $(A,X)$ is a $2$-generated primitive $\mon{\al,\bt}$-axial algebra over $\F$ with axet $X'(k+2k)$. Then $k=1$, $\al+\bt=1$ and either
\begin{enumerate}
\item[$1.$] $A \cong \TC(\al,1-\al)$ for $\al\neq \frac{1}{2}$, or
\item[$2.$] $(\al,\bt)=(\frac{1}{3},\frac{2}{3})$ and either
\begin{enumerate}
    \item[$i)$] $A\cong Q_2(\frac{1}{3},\frac{2}{3}$) if $\F$ has characteristic not equal to $5$, or
    \item[$ii)$] $A\cong Q_2(\frac{1}{3})^\times\oplus \GenG{\id}$ if $\F$ has characteristic equal to $5$.
\end{enumerate}
\end{enumerate}
\proof Follows from Theorem 1.1 in \cite{turner2023skew}, Theorem 1.1, and Theorem 1.2. \qed 
\end{thm}
All the algebras above have been constructed in Section 3 in \cite{turner2023skew}. This classification gives a complete answer to Question 1.2 in \cite{mcinroy2023forbidden} and Problem 6.14 in \cite{mcinroy2022axial}.  

The author would like to thank the anonymous referee for their useful suggestions and comments, particularly in Section $5$ which simplifies the arguments greatly. 
\section{Background}
Since we defined these concepts in the previous paper, we will keep this section short. We recommend the reader to look at \cite{turner2023skew} and/or any citations to get more detail. 
\subsection{Axial Algebras}
Axial algebras are non-associative, commutative algebras with certain properties depending on their eigenspaces. Good introductions to these algebras would be \cite{hall2015universal}, \cite{khasraw2020structure}, and \cite{mcinroy2022axial}.

Let $\F$ be a field and $\mcal{F}$ be a subset of $\F$. A \emph{fusion law} on $\mcal{F}$ is a map $\star: \mcal{F} \times \mcal{F} \rightarrow 2^{\mcal{F}}$ with $2^\mcal{F}$ denoting the power set of $\mcal{F}$. Let  $A$ be a commutative algebra over $\F$. We define the adjoint map for $a\in A$ to be
$\text{ad}_a: A \rightarrow A$, such that $\text{ad}_a(x)=ax$ for all $x\in A$. The set of eigenvalues of $\text{ad}_a$ is denoted by $\Spec{a}$ and for all $\lm \in \F$, $A_\lm(a)$ is the $\lm$-eigenspace of $\text{ad}_a$. Further, for $S\subseteq \F$, we let $A_S(a):=\oplus_{\lm\in S}A_\lm(a)$. 

\begin{defn}
Let $\fustar{F}$ be a fusion law, $A$ be a commutative algebra over $\F$, and $a\in A$. We say that $a$ is a (primitive) $\mcal{F}$-\emph{axis} if it satisfies the following:
\begin{enumerate}
    \item[A$1.$] $a$ is an idempotent; $a^2=a$,
    \item[A$2.$] $a$ is semisimple and $\spec{a}\subseteq \mcal{F}$; that is, $A=A_\mcal{F}(a)$,
    \item[A$3$.] For all $\lm,\mu \in \mcal{F}$, $A_\lm(a)A_\mu(a) \subseteq A_{\lm \star \mu}(a)$, and 
    \item[A$4$.] $A_1(a)=\GenG{a}$.
\end{enumerate}
Further, let $X$ be a set of (primitive) $\mcal{F}$-axes. We call $(A,X)$ a (primitive) $\mcal{F}$-\emph{axial algebra} if $X$ generates $A$.
\end{defn}
\begin{note}
One can define axial algebras by excluding A$4$ and they would be non-primitive. As we will only be concerned with primitive axial algebras, we will assume primitivity throughout. 
\end{note}
We call an axial algebra $A$ of \emph{Monster type} $(\al,\bt)$ if it has the fusion law of $\mon{\al,\bt}$ given in Table \ref{Main Axial} with $\al,\bt \notin \{0,1\}$ and $\al\neq \bt$. 
\begin{table}[!h]
\centering
        \begin{tabular}{|c||c|c|c|c|} 
 \hline
  $\star$& 1 & 0 & $\alpha$ & $\beta$\\ [0.5ex] 
 \hline
 \hline
 1 & 1 &  & $\alpha$ & $\beta$\\ 
 \hline
 0 &  & 0 & $\alpha$ & $\beta$\\
 \hline
 $\alpha$ & $\alpha$ & $\alpha$& 1, 0 & $\beta$\\
 \hline
 $\beta$ & $\beta$ & $\beta$ & $\beta$ & 1, 0, $\alpha$\\
 \hline
 \end{tabular}
    \caption{The fusion law of $\mcal{M}(\alpha, \beta)$.}\label{Main Axial}
\end{table}

For the rest of the paper, we assume the characteristic of $\F$ is not equal to two and $(A,X)$ is a $2$-generated $\mon{\al,\bt}$-axial algebra. For any $\mon{\al,\bt}$-axis $a$, we define $\tu{a}$ to be the automorphism of $A$ which acts as the identity on $A_{\{1,0,\al\}}(a)$ and negative identity on $A_\bt(a)$. We call $\tu{a}$ the \emph{Miyamoto involution} of $a$. Due to A$2$ and A$4$, any $x\in A$ can be written as $x=\mu a + x_0 +x_\al+x_\bt$ with $x_\nu\in A_\nu(a)$ for $\nu\in\{0,\al,\bt\}$. We define the \emph{projection map} to be
\begin{eqnarray*}
        \lm_a : A \rightarrow \F\\
        x \mapsto \mu.
    \end{eqnarray*}
The projection map is a well-defined $\F$-linear map by Lemma 3.4 in \cite{franchi20211}.

\subsection{Axets}
Axets are a recent concept which can be traced back to work in \cite{mcinroy2020expansion}. They have been used in multiple papers; for example \cite{mcinroy2023forbidden} and \cite{mcinroy2022axial} while relating axial algebras with more group theoretic ideas. These axets help build a shape of an axial algebra $A$, which leads to deeper understanding of the structure of $A$.
\begin{defn}
Let $S$ be a group. Suppose $G$ is a group which acts on a set $X$ and there is a map $\tau:X\times S \rightarrow G$, denoted $\tau(x,s)=\tu{x}(s)$. Then $(G,X,\tau)$ is called an $S$-axet if for all $x\in X$, $s,s'\in S$ and $g\in G$, the following properties hold:
\begin{enumerate}
    \item $\tu{x}(s)\in G_x$,
    \item $\tu{x}(ss')=\tu{x}(s)\tu{x}(s')$, and 
    \item $\tu{xg}(s)=\tu{x}(s)^g$.
\end{enumerate}
If it is clear, we denote the axet by $X$.
For each $x\in X$, let $T_x:=\text{Im}(\tu{x})$ which is called the \emph{axial subgroup corresponding to }$x$.
\end{defn}
\begin{note}
We can assume $S$ to be abelian as $T_x\leq Z(G_x)$ and $[S,S]\leq \text{ker}(\tu{x})$ for all $x\in X$. 
\end{note}
For an axet $X$, we call $Y\subseteq X$ $\emph{closed}$ if it is invariant under $T_y$ for all $y\in Y$. For $Z\subseteq X$, we  denote $\GenG{Z}$ to be the smallest closed subset containing $Z$ and we call $\GenG{Z}$ the \emph{closure} of $Z$. Hence $Z$ is closed if and only if $Z=\GenG{Z}$. This paper will be concerned with closed subaxets which are generated by two elements. We denote $C_2=\{e,s\}$ to be the group of order $2$ with identity element $e$.
\begin{exmp}
We will mention some examples of $C_2$-axets that will be used in this paper.
\begin{enumerate}
    \item For $n\geq 2$, denote $X(n)$ to be the vertices of a regular $n$-gon. Let $X=X(n)$ and label the vertices by $a_i$ with $1\leq i\leq n$ in anticlockwise manner. Let $G=D_{2n}$ and for $a\in X$, $\tu{a_i}(e)$ is the identity map and $\tu{a_i}(s)$ is the reflection of $X$ at vertex $a_i$. This is an example of a \emph{regular} axet.
    \item Let $n=4k$ and we look at the $C_2$-axet $X(4k)$. Now identify the opposite vertices of one bipartite half of the polygon denoting this by $X':=X'(k+2k)$. To reiterate, the skew identity can be stated as: for all $i\in \Z$, $a_{2i}=a_{2(i+k)}$. Let $G=D_{4k}$ and $\tau'$ to be the restriction of the domain to $X'\times C_2$. Then $(G,X',\tau')$ is a $C_2$-axet and we call $X'$ \emph{skew}. 
\end{enumerate}
In both examples, the axets are $2$-generated by $\{a_0,a_1\}$. Further, we can explicitly say how $\tu{i}:=\tu{a_i}(s)$ acts: for $j\in \Z$, $a_j^{\tu{i}}=a_{2i-j}$.
\end{exmp}
In Figures \ref{Skew-2} and \ref{Skew-3}, we show $X'(2+4)$ and $X'(3+6)$, where the dotted arrow represents equivalence between two vertices. 

\begin{figure}[!t]
\centering
\begin{minipage}{.5\textwidth}
\centering
\tikzset{every picture/.style={line width=0.75pt}} 

\begin{tikzpicture}[x=0.75pt,y=0.75pt,yscale=-1,xscale=1]

\draw   (408,138.34) -- (383.59,197.27) -- (324.66,221.68) -- (265.73,197.27) -- (241.32,138.34) -- (265.73,79.41) -- (324.66,55) -- (383.59,79.41) -- cycle ;
\draw [color={rgb, 255:red, 232; green, 1; blue, 29 }  ,draw opacity=1 ] [dash pattern={on 0.84pt off 2.51pt}]  (324.66,58) -- (324.66,218.68) ;
\draw [shift={(324.66,221.68)}, rotate = 270] [fill={rgb, 255:red, 232; green, 1; blue, 29 }  ,fill opacity=1 ][line width=0.08]  [draw opacity=0] (8.93,-4.29) -- (0,0) -- (8.93,4.29) -- cycle    ;
\draw [shift={(324.66,55)}, rotate = 90] [fill={rgb, 255:red, 232; green, 1; blue, 29 }  ,fill opacity=1 ][line width=0.08]  [draw opacity=0] (8.93,-4.29) -- (0,0) -- (8.93,4.29) -- cycle    ;
\draw [color={rgb, 255:red, 232; green, 1; blue, 29 }  ,draw opacity=1 ][dash pattern={on 0.84pt off 2.51pt}]  (244.32,138.34) -- (405,138.34) ;
\draw [shift={(408,138.34)}, rotate = 180] [fill={rgb, 255:red, 232; green, 1; blue, 29 }  ,fill opacity=1 ][line width=0.08]  [draw opacity=0] (8.93,-4.29) -- (0,0) -- (8.93,4.29) -- cycle    ;
\draw [shift={(241.32,138.34)}, rotate = 360] [fill={rgb, 255:red, 232; green, 1; blue, 29 }  ,fill opacity=1 ][line width=0.08]  [draw opacity=0] (8.93,-4.29) -- (0,0) -- (8.93,4.29) -- cycle    ;

\draw (320,225) node [anchor=north west][inner sep=0.75pt]   [align=left] {$\displaystyle a_0$};
\draw (320,40) node [anchor=north west][inner sep=0.75pt]   [align=left] {$\displaystyle a_0$};
\draw (385,195.27) node [anchor=north west][inner sep=0.75pt]   [align=left] {$\displaystyle a_{-3}$};
\draw (250,70) node [anchor=north west][inner sep=0.75pt]   [align=left] {$\displaystyle a_1$};
\draw (409,133) node [anchor=north west][inner sep=0.75pt]   [align=left] {$\displaystyle a_2$};
\draw (225,133) node [anchor=north west][inner sep=0.75pt]   [align=left] {$\displaystyle a_2$};
\draw (383,70) node [anchor=north west][inner sep=0.75pt]   [align=left] {$\displaystyle a_{-1}$};
\draw (252,195) node [anchor=north west][inner sep=0.75pt]   [align=left] {$\displaystyle a_{3}$};

\end{tikzpicture}

    \caption{$X'(2+4)$}
    \label{Skew-2}
\end{minipage}%
\begin{minipage}{.5\textwidth}
\centering
\scalebox{0.82}{
\tikzset{every picture/.style={line width=0.75pt}} 

\begin{tikzpicture}[x=0.75pt,y=0.75pt,yscale=-1,xscale=1]

\draw   (434.33,145.67) -- (420.44,197.5) -- (382.5,235.44) -- (330.67,249.33) -- (278.83,235.44) -- (240.89,197.5) -- (227,145.67) -- (240.89,93.83) -- (278.83,55.89) -- (330.67,42) -- (382.5,55.89) -- (420.44,93.83) -- cycle ;
\draw [color={rgb, 255:red, 255; green, 2; blue, 2 }  ,draw opacity=1 ] [dash pattern={on 0.84pt off 2.51pt}]  (330.67,45) -- (330.67,246.33) ;
\draw [shift={(330.67,249.33)}, rotate = 270] [fill={rgb, 255:red, 255; green, 2; blue, 2 }  ,fill opacity=1 ][line width=0.08]  [draw opacity=0] (8.93,-4.29) -- (0,0) -- (8.93,4.29) -- cycle    ;
\draw [shift={(330.67,42)}, rotate = 90] [fill={rgb, 255:red, 255; green, 2; blue, 2 }  ,fill opacity=1 ][line width=0.08]  [draw opacity=0] (8.93,-4.29) -- (0,0) -- (8.93,4.29) -- cycle    ;
\draw [color={rgb, 255:red, 255; green, 0; blue, 0 }  ,draw opacity=1 ] [dash pattern={on 0.84pt off 2.51pt}]  (243.49,196) -- (417.85,95.33) ;
\draw [shift={(420.44,93.83)}, rotate = 150] [fill={rgb, 255:red, 255; green, 0; blue, 0 }  ,fill opacity=1 ][line width=0.08]  [draw opacity=0] (8.93,-4.29) -- (0,0) -- (8.93,4.29) -- cycle    ;
\draw [shift={(240.89,197.5)}, rotate = 330] [fill={rgb, 255:red, 255; green, 0; blue, 0 }  ,fill opacity=1 ][line width=0.08]  [draw opacity=0] (8.93,-4.29) -- (0,0) -- (8.93,4.29) -- cycle    ;
\draw [color={rgb, 255:red, 255; green, 0; blue, 0 }  ,draw opacity=1 ] [dash pattern={on 0.84pt off 2.51pt}]  (243.49,95.33) -- (417.85,196) ;
\draw [shift={(420.44,197.5)}, rotate = 210] [fill={rgb, 255:red, 255; green, 0; blue, 0 }  ,fill opacity=1 ][line width=0.08]  [draw opacity=0] (8.93,-4.29) -- (0,0) -- (8.93,4.29) -- cycle    ;
\draw [shift={(240.89,93.83)}, rotate = 30] [fill={rgb, 255:red, 255; green, 0; blue, 0 }  ,fill opacity=1 ][line width=0.08]  [draw opacity=0] (8.93,-4.29) -- (0,0) -- (8.93,4.29) -- cycle    ;

\draw (322,27) node [anchor=north west][inner sep=0.75pt]   [align=left] {$\displaystyle a_{0}$};
\draw (322,255) node [anchor=north west][inner sep=0.75pt]   [align=left] {$\displaystyle a_{0}$};
\draw (225,82) node [anchor=north west][inner sep=0.75pt]   [align=left] {$\displaystyle a_{2}$};
\draw (216,199) node [anchor=north west][inner sep=0.75pt]   [align=left] {$\displaystyle a_{-2}$};
\draw (422.44,195) node [anchor=north west][inner sep=0.75pt]   [align=left] {$\displaystyle a_{2}$};
\draw (423,84) node [anchor=north west][inner sep=0.75pt]   [align=left] {$\displaystyle a_{-2}$};
\draw (385,45) node [anchor=north west][inner sep=0.75pt]   [align=left] {$\displaystyle a_{-1}$};
\draw (260,45) node [anchor=north west][inner sep=0.75pt]   [align=left] {$\displaystyle a_{1}$};
\draw (209,139) node [anchor=north west][inner sep=0.75pt]   [align=left] {$\displaystyle a_{3}$};
\draw (436,139) node [anchor=north west][inner sep=0.75pt]   [align=left] {$\displaystyle a_{-3}$};
\draw (267,235) node [anchor=north west][inner sep=0.75pt]   [align=left] {$\displaystyle a_{5}$};
\draw (385,235) node [anchor=north west][inner sep=0.75pt]   [align=left] {$\displaystyle a_{-5}$};
\end{tikzpicture}}
\caption{$X'(3+6)$}
\label{Skew-3}
\end{minipage}
\end{figure}
\subsection{The Foundations}
Like in \cite{turner2023skew}, we let $X=\{a_0,a_1\}$ and $A=\GenA{X}$ be an axial algebra of Monster type $(\al,\bt)$. Set $\rho=\tu{a_0}\tu{a_1}$ and let 
\[ a_{2i}=a_0^{\rho^i} \text{ and } a_{2i+1}=a_1^{\rho^i}\]
for all $i\in \Z$. Each $a_j$ is an axis, since $\rho$ is an automorphism, and for ease of notation, we let $\tu{j}:=\tu{a_j}$. For $r\in \N$, let $s_{i,r}:=a_ia_{i+r}-\bt(a_i+a_{i+r})$. We now state Lemmas 6.1 and 6.2 in \cite{franchi20211} (Note that their paper has been restructured since our first paper thus the numberings are different). These will be useful for when we look at $k$ being even. 
\begin{lem}\label{grp s}
For $r\in\N$ and $i\in\Z$, we have $s_{i,r}$ is fixed by the group $\GenG{\tu{i},\tu{i+r}}$.
\proof Follows from the definition of the Miyamoto map. \qed
\end{lem}

\begin{lem}\label{mod s}
For $r\in \N$ and $i,j\in \Z$ such that $i\equiv j\Mod{r}$, we have $s_{i,r}=s_{j,r}$.
\proof Follows from Lemma \ref{grp s}. \qed
\end{lem}
For $i\in \Z$, we define $\lm_i:=\lm_{a_0}(a_i)$ and the following constants: 
\[ \gm_i:=\bt-\lm_i \text{ and } \ep_i:=(1-\al)\lm_i-\bt.\]
Finally, by the semisimplicity of $a_0$, we can write $a_i$ in terms of eigenvectors of $a_0$. We have 
\[ a_i=\lm_i a_0 +u_i+v_i+w_i\]
where $u_i\in A_0(a_0)$, $v_i\in A_\al(a_0)$ and $w_i\in A_\bt(a_0)$. 
\begin{lem}[\protect{\cite[Lemma 6.4]{franchi20211}}]\label{eigen}
For $i\in \Z$, we have
\begin{itemize}
    \item $u_i=\frac{1}{\al}(\ep_i a_0 +\frac{1}{2}(\al-\bt)(a_i+a_{-i})-s_{0,i})$,
    \item $v_i=\frac{1}{\al}(\gm_i a_0 +\frac{1}{2}\bt(a_i+a_{-i})+s_{0,i})$, and
    \item $w_i=\frac{1}{2}(a_i-a_{-i})$.
\end{itemize}
\end{lem}

\section{The Odd Case}
In this section, we will fix $k>1$ and odd. Our method is to create two subalgebras of $A$. First, a skew subalgebra with axet $X'(1+2)$ as shown in the previous paper \cite{turner2023skew}. Second, a symmetric subalgebra with axet $X(2k)$. With these two subalgebras, our possible cases are heavily reduced.  For this section, $(A,X)$ is an axial algebra of Monster type $(\al,\bt)$ with axet $X'(k+2k)$ for $k>1$ and odd. 

\begin{prop}[\protect{\cite[Proposition 8.1]{turner2023skew}}]
\label{odd k}
In $X'(k+2k)$, there exists a subaxet isomorphic to $X'(1+2)$. 
\end{prop}
We illustrate the above Proposition in Figure \ref{1subaxet} taking $k=3$. The next proposition is an extension of Corollary 8.2 in \cite{turner2023skew}.
\begin{figure}[ht]
    \centering
  
\scalebox{.7}{
\tikzset{every picture/.style={line width=0.75pt}} 

\begin{tikzpicture}[x=0.75pt,y=0.75pt,yscale=-1,xscale=1]

\draw   (461,192.5) -- (444.45,254.25) -- (399.25,299.45) -- (337.5,316) -- (275.75,299.45) -- (230.55,254.25) -- (214,192.5) -- (230.55,130.75) -- (275.75,85.55) -- (337.5,69) -- (399.25,85.55) -- (444.45,130.75) -- cycle ;
\draw  [color={rgb, 255:red, 74; green, 144; blue, 226 }  ,draw opacity=1 ] (337.5,69) -- (461.25,192.75) -- (337.5,316.5) -- (213.75,192.75) -- cycle ;
\draw [color={rgb, 255:red, 255; green, 0; blue, 0 }  ,draw opacity=1 ] [dash pattern={on 0.84pt off 2.51pt}]  (337.5,72) -- (337.5,313) ;
\draw [shift={(337.5,316)}, rotate = 270] [fill={rgb, 255:red, 255; green, 0; blue, 0 }  ,fill opacity=1 ][line width=0.08]  [draw opacity=0] (8.93,-4.29) -- (0,0) -- (8.93,4.29) -- cycle    ;
\draw [shift={(337.5,69)}, rotate = 90] [fill={rgb, 255:red, 255; green, 0; blue, 0 }  ,fill opacity=1 ][line width=0.08]  [draw opacity=0] (8.93,-4.29) -- (0,0) -- (8.93,4.29) -- cycle    ;

\draw (331,53) node [anchor=north west][inner sep=0.75pt]   [align=left] {$\displaystyle a_{0}$};
\draw (331,320) node [anchor=north west][inner sep=0.75pt]   [align=left] {$\displaystyle a_{0}$};
\draw (196,185) node [anchor=north west][inner sep=0.75pt]   [align=left] {$\displaystyle a_{3}$};
\draw (465,185) node [anchor=north west][inner sep=0.75pt]   [align=left] {$\displaystyle a_{-3}$};
\end{tikzpicture}}
    \caption{Representation of $X'(3+6)$ having a subaxet of $X'(1+2)$. }
    \label{1subaxet}
\end{figure}
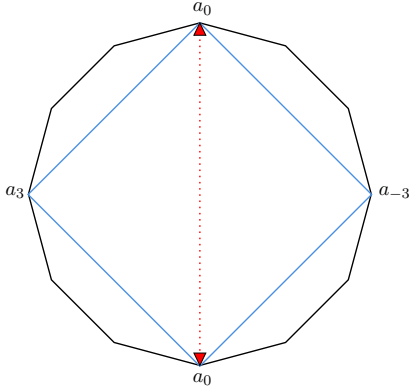
\begin{prop}\label{odd skew}
Let $m\in \Z$ and odd, $S_m:=\GenA{a_m,a_{m+k}}$ and $U_m:=\GenA{a_m,a_{m+2k}}$. 
\begin{itemize}
    \item[$1$.] Then $S_m$ is a skew axial algebra with axet $X'(1+2)$ and is isomorphic to $\TC(\al,1-\al)$, $Q_2(\frac{1}{3},\frac{2}{3})$, or $Q_2(\frac{1}{3})^\times \oplus \GenG{\id}$. Further, $\al+\bt=1$. 
    \item[$2$.] We have that 
    \begin{itemize}
    \item[$(i)$] $U_m\cong \TC(\al)$ if and only if $S_m\cong \TC(\al,1-\al)$ with $\al \neq -1$,
    \item[$(ii)$] $U_m\cong \TC(-1)^\times$ if and only if $S_m\cong \TC(-1,2)$, and 
    \item[$(iii)$] $U_m\cong \TB$ if and only if $S_m\cong Q_2(\frac{1}{3},\frac{2}{3})$ with field characteristic not equal to $5$ or $S_m\cong Q_2(\frac{1}{3})^\times \oplus \GenG{\id}$ with field characteristic equal to $5$.
\end{itemize}
\item[$3$.] For $n\in \Z$ and odd, $U_m\cong U_n$ and $S_m \cong S_n$. 
\end{itemize}
\proof Fix $m\in \Z$ and odd. 
\begin{itemize}
    \item[$1$.] Let $a=a_{m+k}$, $b=a_m$ and $c=a_{m+2k}$. We have $b^{\tu{a}}=c$ and $a^{\tu{b}}=a^{\tu{c}}=a$ and so $S_m=\GenA{a,b}$ is skew with axet $X'(1+2)$. By Corollary 8.2 in \cite{turner2023skew}, $\al+\bt=1$ and claim 1 follows from Theorem 1.1 in \cite{turner2023skew}.  
    
    \item[$2$.] Note that $U_m$ is a subalgebra of $S_m$. This follows from looking at the structure of each possible $S_m$ which are explicitly constructed in Section 3 of \cite{turner2023skew}. 
     
    \item[$3$.] Suppose $n\in \Z$ and is odd, and  $l=\frac{n+m}{2}$. We have that $a_n^{\tu{l}}=a_m$ and $a_{n+2k}^{\tu{l}}=a_{m-2k}=a_{m+2k}$. Hence $U_n^{\tu{l}}=U_m$. As $\tu{l}$ is an automorphism for the whole algebra, it produces a isomorphism between $U_n$ and $U_m$. By $2$ and assuming $U_m\not\cong \TB$, $S_n\cong S_m$. If $U_m\cong \TB$, then  $S_m$ could be two options. As $Q_2(\frac{1}{3},\frac{2}{3})$ and $Q_2(\frac{1}{3})^\times \oplus \GenG{\id}$ do not exist in the same characteristic, $S_m$ can only be one option and $S_m\cong S_n$.
\end{itemize} \qed
\end{prop}

Let $L:=\GenA{a_{-1},a_1}$. It is a $2$-generated axial algebra and is of Monster type $(\al,\bt)$. Remember that a $2$-generated axial algebra, $\GenA{a,b}$, is \emph{symmetric} if there is an involutionary automorphism $f$ such that $a^f=b$. We have that $\tu{0}$ switches $a_1$ and $a_{-1}$ thus $L$ is symmetric. Finally, $L$ has axet $X(2k)$ since it only has the odd axes in its axet. We have illustrated this subaxet for $k=3$ in Figure \ref{3 subaxet}. To simplify notation, we let $b_i:=a_{2i-1}$ for $i\in \Z$ making $L=\GenA{b_0,b_1}$. We now recall the classification of all symmetric axial algebras of Monster type. 

\begin{thm}[\cite{yabe2023classification, franchi2022classifying, franchi2022quotients}]\label{Symm Thm}
    Let $V$ be a symmetric $2$-generated $\mon{\al,\bt}$-axial algebra. Then $V$ is isomorphic to one of the following:
    \begin{enumerate}
        \item [$1.$] an axial algebra of Jordan type $\al$ or $\bt$;
        \item [$2.$] a quotient of the Highwater algebra $\hw$, or its characteristic 5 cover $\hwc$, where $(\al,\bt)=(2,\frac{1}{2})$; or
        \item[$3.$] a quotient of one of the following algebras (given in Table $2$ in \cite{yabe2023classification}):
        \begin{itemize}
    \item [$(a)$] $\TA(\al,\bt)$, $\FA(\frac{1}{4},\bt)$, $\FB(\al,\frac{\al^2}{2})$, $\FJ(2\bt,\bt)$, $\FY(\frac{1}{2},\bt)$, $\FY(\al,\frac{1-\al^2}{2})$,\newline
    $\CA(\al,\frac{5\al-1}{2})$, $\SA(\al, -\frac{\al^2}{4(2\al-1)})$, $\SJ(2\bt,\bt)$, and $\SY(\frac{1}{2},2)$;
    \item [$(b)$] $\IY{3}(\al,\frac{1}{2},\mu)$ and $\IY{5}(\al,\frac{1}{2})$.
\end{itemize}
\end{enumerate}
where $(a)$ are the algebras with fixed finite axets $X(n)$ with $n$ being the number before the capital letter\footnote{The capital letter comes from either the algebra being a generalisation of a Norton-Sakuma algebra or from the discoverer (Joshi and Yabe)} and $(b)$ are the algebras which have varying axets depending on an extra parameter and/or field characteristic.  
\end{thm}
\begin{note}
We use notation from \cite{mcinroy2023forbidden} for the algebras in $3$. This focuses on the axet rather than the axial dimension of each algebra. Further,  the possible axets of $\IY{3}(\al,\frac{1}{2},\mu)$ and $\IY{5}(\al,\frac{1}{2})$ have been found in the same paper.
\end{note}
\begin{rem}
An axial algebra cannot be symmetric and skew at the same time. Suppose so, then the flip automorphism would map odd axes to even axes and vice versa. Since there are $k$ even axes and $2k$ odd axes, $f$ is not bijective producing a contradiction. 
\end{rem}
\begin{figure}
    \centering
\tikzset{every picture/.style={line width=0.75pt}} 
\begin{minipage}{.5\textwidth}
\centering
\scalebox{.8}{
\begin{tikzpicture}[x=0.75pt,y=0.75pt,yscale=-1,xscale=1]

\draw   (434.33,145.67) -- (420.44,197.5) -- (382.5,235.44) -- (330.67,249.33) -- (278.83,235.44) -- (240.89,197.5) -- (227,145.67) -- (240.89,93.83) -- (278.83,55.89) -- (330.67,42) -- (382.5,55.89) -- (420.44,93.83) -- cycle ;
\draw [color={rgb, 255:red, 255; green, 0; blue, 0 }  ,draw opacity=1 ] [dash pattern={on 0.84pt off 2.51pt}]  (330.67,45) -- (330.67,246.33) ;
\draw [shift={(330.67,249.33)}, rotate = 270] [fill={rgb, 255:red, 255; green, 0; blue, 0 }  ,fill opacity=1 ][line width=0.08]  [draw opacity=0] (8.93,-4.29) -- (0,0) -- (8.93,4.29) -- cycle    ;
\draw [shift={(330.67,42)}, rotate = 90] [fill={rgb, 255:red, 255; green, 0; blue, 0 }  ,fill opacity=1 ][line width=0.08]  [draw opacity=0] (8.93,-4.29) -- (0,0) -- (8.93,4.29) -- cycle    ;
\draw [color={rgb, 255:red, 255; green, 0; blue, 0 }  ,draw opacity=1 ] [dash pattern={on 0.84pt off 2.51pt}]  (243.49,196) -- (417.85,95.33) ;
\draw [shift={(420.44,93.83)}, rotate = 150] [fill={rgb, 255:red, 255; green, 0; blue, 0 }  ,fill opacity=1 ][line width=0.08]  [draw opacity=0] (8.93,-4.29) -- (0,0) -- (8.93,4.29) -- cycle    ;
\draw [shift={(240.89,197.5)}, rotate = 330] [fill={rgb, 255:red, 255; green, 0; blue, 0 }  ,fill opacity=1 ][line width=0.08]  [draw opacity=0] (8.93,-4.29) -- (0,0) -- (8.93,4.29) -- cycle    ;
\draw [color={rgb, 255:red, 255; green, 0; blue, 0 }  ,draw opacity=1 ] [dash pattern={on 0.84pt off 2.51pt}]  (243.49,95.33) -- (417.85,196) ;
\draw [shift={(420.44,197.5)}, rotate = 210] [fill={rgb, 255:red, 255; green, 0; blue, 0 }  ,fill opacity=1 ][line width=0.08]  [draw opacity=0] (8.93,-4.29) -- (0,0) -- (8.93,4.29) -- cycle    ;
\draw [shift={(240.89,93.83)}, rotate = 30] [fill={rgb, 255:red, 255; green, 0; blue, 0 }  ,fill opacity=1 ][line width=0.08]  [draw opacity=0] (8.93,-4.29) -- (0,0) -- (8.93,4.29) -- cycle    ;

\draw (322,27) node [anchor=north west][inner sep=0.75pt]   [align=left] {$\displaystyle a_{0}$};
\draw (322,255) node [anchor=north west][inner sep=0.75pt]   [align=left] {$\displaystyle a_{0}$};
\draw (225,82) node [anchor=north west][inner sep=0.75pt]   [align=left] {$\displaystyle a_{2}$};
\draw (216,199) node [anchor=north west][inner sep=0.75pt]   [align=left] {$\displaystyle a_{-2}$};
\draw (422.44,195) node [anchor=north west][inner sep=0.75pt]   [align=left] {$\displaystyle a_{2}$};
\draw (423,84) node [anchor=north west][inner sep=0.75pt]   [align=left] {$\displaystyle a_{-2}$};
\draw (383,43) node [anchor=north west][inner sep=0.75pt]   [align=left] {$\displaystyle a_{-1}$};
\draw (265,43) node [anchor=north west][inner sep=0.75pt]   [align=left] {$\displaystyle a_{1}$};
\draw (209,137) node [anchor=north west][inner sep=0.75pt]   [align=left] {$\displaystyle a_{3}$};
\draw (436,137) node [anchor=north west][inner sep=0.75pt]   [align=left] {$\displaystyle a_{-3}$};
\draw (267,234) node [anchor=north west][inner sep=0.75pt]   [align=left] {$\displaystyle a_{5}$};
\draw (383,234) node [anchor=north west][inner sep=0.75pt]   [align=left] {$\displaystyle a_{-5}$};

\end{tikzpicture}}
\end{minipage}%
\begin{minipage}{.5\textwidth}
\centering

\scalebox{.75}{\tikzset{every picture/.style={line width=0.75pt}} 

\begin{tikzpicture}[x=0.75pt,y=0.75pt,yscale=-1,xscale=1]

\draw   (433.33,147.83) -- (381.08,238.33) -- (276.58,238.33) -- (224.33,147.83) -- (276.58,57.33) -- (381.08,57.33) -- cycle ;

\draw (381,42) node [anchor=north west][inner sep=0.75pt]   [align=left] {$\displaystyle b_{0}$};
\draw (262,42) node [anchor=north west][inner sep=0.75pt]   [align=left] {$\displaystyle b_{1}$};
\draw (208,140) node [anchor=north west][inner sep=0.75pt]   [align=left] {$\displaystyle b_{2}$};
\draw (435,140) node [anchor=north west][inner sep=0.75pt]   [align=left] {$\displaystyle b_{-1}$};
\draw (262,235) node [anchor=north west][inner sep=0.75pt]   [align=left] {$\displaystyle b_{3}$};
\draw (381,235) node [anchor=north west][inner sep=0.75pt]   [align=left] {$\displaystyle b_{-2}$};
\draw [color={rgb, 255:red, 255; green, 0; blue, 0 }  ,draw opacity=1 ] [dash pattern={on 0.84pt off 2.51pt}]  (328.83,40) -- (328.83,250) ;
\draw (323.83,25.35) node [anchor=north west][inner sep=0.75pt]   [align=left] {$\displaystyle \tu{0}$};
\end{tikzpicture}}
\end{minipage}
\caption{Representation of $X'(3+6)$ having a subaxet of $X(6)$ which is symmetric.}
\label{3 subaxet}
\end{figure}
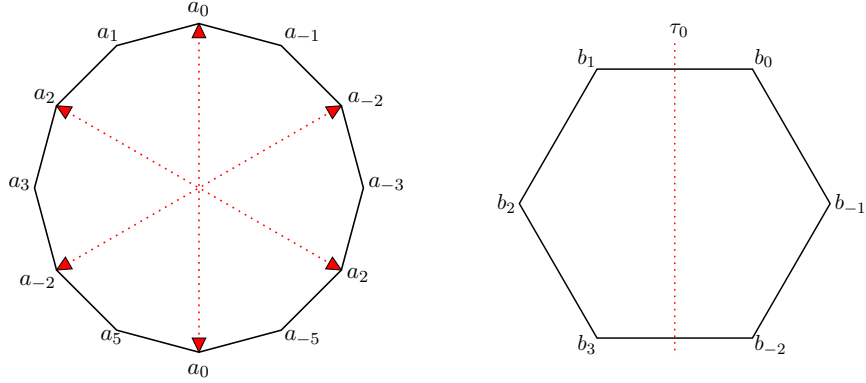
\begin{prop}\label{Odd Prop}
We have $k=3$ and $L$ is isomorphic to $\SA(\al,-\frac{\al^2}{4(2\al-1)})$, $\SJ(2\bt,\bt)$ or their quotients, with $\al+\bt=1$. 
\proof From the previous discussion, $L$ is a $2$-generated symmetric $\mon{\al,\bt}$-axial algebra with axet $X(2k)$. By Proposition \ref{odd skew}, we have $\al+\bt=1$ and $\al$ nor $\bt$ can be equal to $\frac{1}{2}$. Thus $L$ cannot be a $\hw$, $\hwc$, $\IY{3}(\al,\frac{1}{2},\mu)$, $\IY{5}(\al,\frac{1}{2})$ or any of their quotients. If $L$ was of Jordan type $\bt\neq\frac{1}{2}$, its axet would be $X(2)$ or $X(3)$ and cannot happen. Looking at the remaining algebras which have axet $X(2k)$, the largest is $X(6)$ thus $k=3$. As $\al\neq \frac{1}{2}$, we have $L$ cannot be isomorphic to $\SY(\frac{1}{2},2)$ or its quotients.  \qed
\end{prop}
We now focus on $L$ being isomorphic to $\SJ(\frac{2}{3},\frac{1}{3})$ or to its quotient $\SJ(\frac{2}{3},\frac{1}{3})^\times$ if the characteristic is equal to $5$. As $\al=2\bt$, we can use work done in \cite{franchi20212}.
\begin{thm}[\protect{\cite[Theorem 1.1]{franchi20212}}]\label{2beta}
Let $V$ be a $2$-generated $\mon{2\bt,\bt}$-axial algebra. Then $V$ is symmetric or $V$ is either isomorphic to $\TC(\frac{2}{3},\frac{1}{3})$, $Q_2(\bt)$, or $Q_2(-\frac{1}{2})^\times$. 
\end{thm}
\begin{lem}\label{6J lem}
We have $L\not\cong \SJ(\frac{2}{3},\frac{1}{3})$ or $L\not\cong \SJ(\frac{2}{3},\frac{1}{3})^\times$. 
\proof Suppose $L\cong \SJ(\frac{2}{3},\frac{1}{3})$ or $L\cong \SJ(\frac{2}{3},\frac{1}{3})^\times$. As $L\leq A$, $A$ must be at least $7$-dimensional and has fusion law $\mon{2\bt,\bt}$. By Theorem \ref{2beta}, as $A$ has dimension greater than $4$, $A$ is symmetric producing a contradiction. \qed 
\end{lem}
We now focus on $L$ being isomorphic to $\SA(\al,-\frac{\al^2}{4(2\al-1)})$ or one of its quotients. 

\begin{lem}\label{6A quot lem}
We have that $L$ is not isomorphic to any non-trivial quotient of $\SA(\al,-\frac{\al^2}{4(2\al-1)})$. 
\proof By \cite{yabe2023classification}, the only symmetric quotients of $\SA(\al,-\frac{\al^2}{4(2\al-1)})$ are
\begin{enumerate}
    \item $\SA(\frac{2}{3},-\frac{1}{3})^\times$,
    \item $\SA(\frac{1\pm \sqrt{97}}{24},\frac{53\pm 5\sqrt{97}}{192})^\times$ with characteristic not equal to $11$, or
    \item $\SA(2,-4)^\times$ with characteristic equal to $11$.
\end{enumerate}
For the first case, assuming $\al+\bt=1$, we get that $1=3$ and contradicts the characteristic of $\F$. For the second case,  more arithmetic is involved. Invoking $\al+\bt=1$, we get
\begin{eqnarray*}
    1 &=&\frac{1\pm \sqrt{97}}{24}+\frac{53\pm 5\sqrt{97}}{192} = \frac{1}{192}\left(8\pm 8\sqrt{97}+53\pm 5\sqrt{97}\right)\\
    &=& \frac{1}{192}\left(61\pm 13\sqrt{97}\right)
\end{eqnarray*} 
Therefore $192 = 61 \pm 13\sqrt{97}$. This produces $131=\pm 13\sqrt{97}$. Squaring both sides, we get $17161=16393$ thus $768=0$. Since $768=2^8\times 3$, we get the characteristic is equal to $2$ or $3$ producing a contradiction. For the third case, we have $\al+\bt=2-4=-2\neq 1$ in characteristic $11$ and generates the final contradiction. \qed
\end{lem}
In Table \ref{Mult Table}, we have stated the algebra multiplication of $\SA(\al,-\frac{\al^2}{4(2\al-1)})$ taken from \cite{mcinroy2023forbidden}. Notice that we have included $\gm=\frac{\al}{8(2\al-1)}$ to simplify some of the expressions.
\begin{table}[h]
\centering
\scalebox{0.77}{
\begin{tabular}{|c|c|c|} 
\hline
Type & Basis & Multiplication\\
\hline
\hline
$\SA(\al,-\frac{\al^2}{4(2\al-1)})$ & $b_{-2}$, ..., $b_3$, & $b_ib_{i+1}=\frac{\bt}{2}(b_i+b_{i+1}-b_{i+2}-b_{i+3}-b_{i-1}-b_{i-2}+c+z)$\\
$\gm=\frac{\al}{8(2\al-1)}$& $c$, $z$ & $b_ib_{i+2}=\frac{\al}{4}(b_i+b_{i+2})+2\gm(3\al-1) b_{i+4}-\gm(5\al-2) z$\\
& & $b_ib_{i+3}=\frac{\al}{2}(b_i+b_{i+3}-c)$, $b_ic = \frac{\al}{2}(b_i+c-b_{i+3})$\\
& & $b_iz =2\gm(3\al-2)(2b_i -b_{i-2}-b_{i+2}+z)$\\
& & $c^2=c$ , $cz=0$, $z^2 = \frac{2\gm}{\al}(\al+2)(3\al-2)z$\\
\hline
 \end{tabular}}
    \caption{The multiplication of $\SA(\al,-\frac{\al^2}{4(2\al-1)})$.}\label{Mult Table}
\end{table}

We now assume $L\cong \SA(\al,-\frac{\al^2}{4(2\al-1)})$. Let $m\in \Z$ and be odd. Looking at the multiplication of $\SA(\al,-\frac{\al^2}{4(2\al-1)})$, we have $U_m=\GenA{a_m,a_{m+6}}=\GenA{b_n, b_{n+3}}\cong \TC(\al)$ for $n=2m-1$. By Proposition \ref{odd skew}, $S_m\cong \TC(\al,1-\al)$ with $\al\neq -1$ for all $m$. As $U_m= S_m$, we have $a_{m+3}\in L$ for all $m$. Hence all the axes in the axet, in particular $a_0$ and $a_1$, are in $L$ and so $L=A$.  

\begin{lem}\label{6A lem}
We have $A\not\cong \SA(\al,-\frac{\al^2}{4(2\al-1)})$.
\proof 
By the above, $\al\neq -1$ and $S_3=U_3=\GenA{a_3,a_{-3}}=\GenA{b_2,b_{-1}}\cong \TC(\al)$, thus we have three possible $\jor{\bt}$-axes for $a_0$. These are $\id-b_2$, $\id-b_{-1}$ or $\id-c$, with $\id=\frac{1}{\al+1}(b_2+b_{-1}+c)$ being the identity of $U_3$. The first two options cannot happen as their Miyamoto involution fixes $b_2$ and $b_{-1}$ respectively, which $\tu{0}$ does not do. Hence
$$a_0=\id-c = \frac{1}{\al+1}(b_2+b_{-1}-\al c).$$  
We have that $a_1-a_{-1}=b_1-b_0$ is a $\bt$-eigenvector of $a_0$ and  
\begin{eqnarray}
    (\al+1)\bt(b_1-b_0)&=& (\al+1)a_0(b_1-b_0)\nonumber\\
    &=& (b_2+b_{-1}-\al c)(b_1-b_0)\nonumber\\
    &=&b_1b_2+b_{-1}b_1-\al c b_1 -b_0b_2-b_{-1}b_0+\al cb_0\nonumber\\
    &=&\frac{\bt}{2}b_1+\frac{\bt}{2}b_2-\frac{\bt}{2}b_0-\frac{\bt}{2}b_3-\frac{\bt}{2}b_{-1}-\frac{\bt}{2}b_{-2}+\frac{\bt}{2}c\nonumber\\
    &+&\frac{\bt}{2}z+\frac{\al}{4}b_{-1} +\frac{\al}{4}b_1+2(3\al-1)\gm b_3-(5\al-2)\gm z\nonumber\\
    &-&\frac{\al^2}{2}c -\frac{\al^2}{2}b_1+\frac{\al^2}{2}b_{-2}\nonumber\\
    &-&\frac{\al}{4}b_0 -\frac{\al}{4}b_2-2(3\al-1)\gm b_{-2}+(5\al-2)\gm z\nonumber\\
    &-&\frac{\bt}{2}b_{-1}-\frac{\bt}{2}b_0+\frac{\bt}{2}b_1+\frac{\bt}{2}b_2+\frac{\bt}{2}b_3+\frac{\bt}{2}b_{-2}-\frac{\bt}{2}c\nonumber\\
    &-&\frac{\bt}{2}z+\frac{\al^2}{2}c+\frac{\al^2}{2}b_0-\frac{\al^2}{2}b_3.\label{6A eqn}
\end{eqnarray}
Looking at the $b_2$ coefficient of Equation \eqref{6A eqn}, we have
\[ 0 = \frac{\bt}{2}-\frac{\al}{4}+\frac{\bt}{2} = \bt-\frac{\al}{4}.\]
Therefore $\al=4\bt$ and equivalently
\begin{eqnarray*}
    \al=4\left(\frac{-\al^2}{4(2\al-1)}\right)=\frac{-\al^2}{(2\al-1)}.
\end{eqnarray*}
Solving for $\al$, we either have $\al=0$ or $\al=\frac{1}{3}$. As the former is a contradiction on the fusion law, we assume $\al=\frac{1}{3}$. For the $b_{-2}$ coefficient of Equation \eqref{6A eqn}, observe that
\[ 0 = -\frac{\bt}{2}+\frac{\al^2}{2}-2(3\al-1)\gm+\frac{\bt}{2}=\frac{\al^2}{2}-2(3\al-1)\gm=\frac{1}{18}.\]
This is an obvious contradiction. \qed
\end{lem}
\proof[Proof of Theorem $1.1$]
By Proposition \ref{Odd Prop}, we have $k=3$, and $L\cong \SA(\al,\bt)$, $L\cong \SJ(\al,\bt)$ or any quotient with $\al+\bt=1$. By Lemmas \ref{6J lem}, \ref{6A quot lem}, and \ref{6A lem}, all cases are exhausted and so $A$ cannot exist. \qed 

\section{The Even Case}
We now look at when $k$ is even. As we cannot relate even skew axets with $X'(1+2)$, we need to change our approach. We will do this in two steps. First, we will restrict the problem to $k=2^n$ for $n \in \N$ and it has a similar proof to Proposition \ref{odd k}. With this result, we will look at $X'(2^n + 2^{n+1})$ and show that no axial algebra, of conditions that we desire, can have that axet. Finally, we will relate this to all even skew axets. 

\begin{no}
    For this section, we will let $k=mq$ for $m,q \in \N$ with $m$ odd and $q=2^n$ for some $n\in \N$.
\end{no}
To remind the reader, 
\[X(4k)=\{a_j \; | \; a_j=a_{j+4k} \;\forall j \in \Z\}\]
and $X'(k+2k)$ is a quotient set of $X(4k)$ with the extra identity of $a_{2i}=a_{2(k+i)}$ for all $i\in \Z$. 
\begin{prop}\label{even reduction}
In $X'(k+2k)$, there is a subaxet isomorphic to $X'(q+2q)$. 
\proof Let $X=\{a_0, a_m\}$ and $Z=\GenG{X}$. We will show that $Z$ is isomorphic to $X'(q+2q)$. Applying $\tu{0}$ and $\tu{m}$, one can see that $Z=\{a_{mi} \; | \; i \in \Z\}$. Since we have $a_j=a_{j+4k}$ for all $j\in \Z$, we can express $Z$ with $i\in \{0,...,4q-1\}$. Thus $|Z|\leq 4q$. With the skew identity, we are still double counting some elements. Let $l\in \{0,...,2(q-1)\}$ and even. We have that $a_{lm}=a_{lm+2k}$ and $|Z|=3q$. As there are $q$ even axis and $2q$ odd axis in $Z$, it is not regular and so $Z$ is isomorphic to $X'(q+2q)$. \qed
\end{prop}
We illustrate the above Proposition for $k=6$ with $q=2$ and $m=3$ in Figure \ref{2subaxet}. For the rest of the section, we assume $A$ to be an axial algebra of Monster type with axet $X'(q+2q)$. Our aim is to show that $A$ cannot exist.
\begin{figure}
    \centering
\scalebox{.66}{\tikzset{every picture/.style={line width=0.75pt}} 

\begin{tikzpicture}[x=0.75pt,y=0.75pt,yscale=-1,xscale=1]

\draw   (463.52,158) -- (458.95,192.68) -- (445.56,225) -- (424.27,252.75) -- (396.52,274.05) -- (364.2,287.43) -- (329.52,292) -- (294.83,287.43) -- (262.52,274.05) -- (234.76,252.75) -- (213.47,225) -- (200.08,192.68) -- (195.52,158) -- (200.08,123.32) -- (213.47,91) -- (234.76,63.25) -- (262.52,41.95) -- (294.83,28.57) -- (329.52,24) -- (364.2,28.57) -- (396.52,41.95) -- (424.27,63.25) -- (445.56,91) -- (458.95,123.32) -- cycle ;
\draw  [color={rgb, 255:red, 74; green, 144; blue, 226 }  ,draw opacity=1 ] (463.52,158) -- (424.27,252.75) -- (329.52,292) -- (234.76,252.75) -- (195.52,158) -- (234.76,63.25) -- (329.52,24) -- (424.27,63.25) -- cycle ;
\draw [color={rgb, 255:red, 255; green, 0; blue, 0 }  ,draw opacity=1 ] [dash pattern={on 0.84pt off 2.51pt}]  (329.52,27) -- (329.52,289) ;
\draw [shift={(329.52,292)}, rotate = 270] [fill={rgb, 255:red, 255; green, 0; blue, 0 }  ,fill opacity=1 ][line width=0.08]  [draw opacity=0] (8.93,-4.29) -- (0,0) -- (8.93,4.29) -- cycle    ;
\draw [shift={(329.52,24)}, rotate = 90] [fill={rgb, 255:red, 255; green, 0; blue, 0 }  ,fill opacity=1 ][line width=0.08]  [draw opacity=0] (8.93,-4.29) -- (0,0) -- (8.93,4.29) -- cycle    ;
\draw [color={rgb, 255:red, 255; green, 0; blue, 0 }  ,draw opacity=1 ] [dash pattern={on 0.84pt off 2.51pt}]  (198.52,158) -- (460.52,158) ;
\draw [shift={(463.52,158)}, rotate = 180] [fill={rgb, 255:red, 255; green, 0; blue, 0 } ,fill opacity=1 ][line width=0.08]  [draw opacity=0] (8.93,-4.29) -- (0,0) -- (8.93,4.29) -- cycle    ;
\draw [shift={(195.52,158)}, rotate = 360] [fill={rgb, 255:red, 255; green, 0; blue, 0 }  ,fill opacity=1 ][line width=0.08]  [draw opacity=0] (8.93,-4.29) -- (0,0) -- (8.93,4.29) -- cycle    ;

\draw (322,8) node [anchor=north west][inner sep=0.75pt]   [align=left] {$\displaystyle a_{0}$};
\draw (322,294) node [anchor=north west][inner sep=0.75pt]   [align=left] {$\displaystyle a_{0}$};
\draw (177,152) node [anchor=north west][inner sep=0.75pt]   [align=left] {$\displaystyle a_{6}$};
\draw (467,152) node [anchor=north west][inner sep=0.75pt]   [align=left] {$\displaystyle a_{6}$};
\draw (219,50) node [anchor=north west][inner sep=0.75pt]   [align=left] {$\displaystyle a_{3}$};
\draw (424,50) node [anchor=north west][inner sep=0.75pt]   [align=left] {$\displaystyle a_{-3}$};
\draw (219,250) node [anchor=north west][inner sep=0.75pt]   [align=left] {$\displaystyle a_{9}$};
\draw (424,250) node [anchor=north west][inner sep=0.75pt]   [align=left] {$\displaystyle a_{-9}$};

\end{tikzpicture}}    
\caption{Representation of $X'(6+12)$ having a subaxet of $X'(2+4)$.}
    \label{2subaxet}
\end{figure}
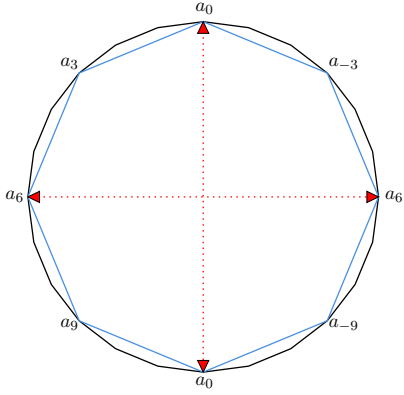
\begin{lem}\label{divider}
Let $t\in \N$ be odd. We have $s_{0,t}=s_{0,2q-t}$. 
\end{lem}
\begin{proof}
As $t$ is odd and $4q$ is a power of $2$, $t$ is a generator of the group $(\Z/4q\Z,+)$. Therefore there exists $l\in \Z$ such that $2q-t\equiv lt\Mod{4q}$ thus $a_{2q-t}=a_{lt}$ and $a_{lt+t}=a_{2q}=a_0$. As $0\equiv lt \Mod{t}$, $s_{0,t}=s_{lt,t}$ by Lemma \ref{mod s}. We have
\[ s_{0,t}=s_{lt,t}=a_{lt}a_{lt+t}-\bt(a_{lt}+a_{lt+t})=a_{2q-t}a_{0}-\bt(a_{2q-t}+a_0)=s_{0,2q-t}.\]
\end{proof}

\begin{lem}\label{even lem}
For $t$ odd, we have $a_t+a_{-t}=a_{2q-t}+a_{-(2q-t)}$.
\proof By Lemma \ref{divider}, $s_{0,t}=s_{0,2q-t}$. Notice that for $r\in \Z$
\[\begin{split}
\lm_{a_0}(s_{0,r})&=\lm_{a_0}(a_0a_r-\bt(a_0+a_r))=\lm_{a_0}(a_0a_r)-\bt\lm_{a_0}(a_0)-\bt\lm_{a_0}(a_r)\\
&=\lm_r-\bt-\bt\lm_r=(1-\bt)\lm_r-\bt 
\end{split}\]
by the properties of the projection map.
We have 
\[(1-\bt)\lm_t-\bt=\lm_{a_0}(s_{0,t})=\lm_{a_0}(s_{0,2q-t})=(1-\bt)\lm_{2q-t}-\bt.\]
Thus $\lm_t=\lm_{2q-t}$. Using $u_i$ and $v_i$ in Lemma \ref{eigen}, we get
\[ \frac{1}{2}(\al-\bt)(a_t+a_{-t}-a_{2q-t}-a_{-(2q-t)})=\al u_t-\al u_{2q-t}\in A_0(a_0)\]
and
\[ \frac{1}{2}\bt(a_t+a_{-t}-a_{2q-t}-a_{-(2q-t)})=\al v_t-\al v_{2q-t}\in A_\al(a_0).\]
An eigenvector can have two different eigenvalues if and only if is equal to $0$. Hence $a_t+a_{-t}=a_{2q-t}+a_{-(2q-t)}$. \qed  
\end{lem}
\begin{prop}\label{power of 2}
We have that $A$ cannot exist. 
\proof Applying $t=q-1$ to Lemma \ref{even lem}, we get
\[a_{q-1}+a_{-(q-1)}=a_{q+1}+a_{-(q+1)}.\]
As $q$ is even, we apply $\tu{\frac{q}{2}}$ to get
\[a_{q-(q-1)}+a_{q+(q-1)}=a_{q-(q+1)}+a_{q+(q+1)}\] which is equivalent to
\begin{equation}\label{even 1}
    a_{1}+a_{2q-1}=a_{-1}+a_{-(2q-1)},
\end{equation}
as $a_{2q+1}=a_{-(2q-1)}$
Applying $t=1$ to Lemma \ref{even lem}, we get
\begin{equation}\label{even 2}
    a_1+a_{-1}=a_{2q-1}+a_{-(2q-1)}.
\end{equation}
Taking the difference of \eqref{even 1} and \eqref{even 2}, we have
\[ 2a_{-1}=2a_{2q-1}.\]
Since the characteristic of $\F$ is not equal to $2$, we get $a_{-1}=a_{2q-1}$. Notice that this cannot happen as $-1$ is odd and this identity is never true. \qed
\end{prop}

\proof[Proof of Theorem $1.2$.] Let $V=\GenA{a_0,a_1}$ be an $\mon{\al,\bt}$-axial algebra with axet $X'(k+2k)$. We have $A=\GenA{a_0,a_m}$ is a subalgebra of $V$. Notice that $A$ is a $2$-generated $\mon{\al,\bt}$-axial algebra and has an axet of $X'(q+2q)$ by Proposition \ref{even reduction}. By Proposition \ref{power of 2}, $A$ cannot exist thus neither can $V$. \qed
\bibliographystyle{abbrv}
\bibliography{references}
\end{document}